\documentclass[a4, 11pt.]{amsart}
\usepackage{amssymb,amscd,latexsym,  hyperref, bm, color }
\usepackage{amsmath}
\usepackage{enumerate}
\usepackage{stmaryrd}  
\usepackage[mathcal]{eucal}
\usepackage{array,float}
\usepackage{longtable}
\usepackage{multirow}
\usepackage{cite}
\usepackage[left=3cm, right=3cm]{geometry}
\usepackage{graphicx}
\usepackage{ amssymb }
\usepackage{ mathrsfs }
\usepackage{amsmath,amscd}
\usepackage{amsfonts}
\usepackage{amsthm}
\usepackage{calrsfs}
\usepackage[utf8]{inputenc}
\usepackage[english]{babel}
\usepackage{tikz-cd}
\usepackage{fixltx2e}
\usepackage{calrsfs}
\usetikzlibrary{cd}

\usetikzlibrary{matrix,arrows}
\usepackage{float}
\restylefloat{table}
\usepackage{imakeidx}
\usepackage{xcolor}
\usepackage[all,color]{xy}

\newcommand{\ovl}{\rotatebox[origin=c]{90}{$\cong$}}

\numberwithin{equation}{section}
\newtheorem{theorem}{Theorem}[section]
\newtheorem{proposition}[theorem]{Proposition}

\newtheorem{corollary}[theorem]{Corollary}
\newtheorem{lemma}[theorem]{Lemma}

\theoremstyle{definition}
\newtheorem{defn}[theorem]{Definition}
\newtheorem*{rem}{Remark}

\title[Lefschetz Theorem Using Real Morse Theory]{Lefschetz Theorem\\Using Real Morse Theory}

\author{Nima Rose Manjila}
\address{ NM: Department of Mathematics,
  Indian Institute of Science Education and Research,
  Dr. Homi Bhabha Road, Pune 411008, India. }

\email{nimarose.manjila@students.iiserpune.ac.in}

\author{A J Parameswaran}
\address{ AJ:  School of Mathematics,
  Tata Institute of Fundamental Research,
 Colaba, Mumbai 400005, India. }
\email{param@math.tifr.res.in}



\begin{document}
\begin{abstract}
Taking~\cite{lamotke1981topology} as the primary reference, we prove some of the important theorems solely using Real Morse theory and exact sequences. The Proof of Lefschetz hyperplane section theorem is a deviation from the conventional one which uses vanishing cycles, thimbles and monodromy. Instead, we prove using simpler machinery.  By making the observation that we can compose the Lefschetz Pencil with a Real Morse function to get a map from the variety to $\mathbb{R}$ which is "close" to being a Real Morse function, we prove the genus formula for plane curves and then we prove the Riemann Hurwitz formula for ramified maps between curves by employing techniques from deformation theory. Lastly, we prove the Lefschetz Hyperplane Section Theorem.
\end{abstract}
\maketitle
\tableofcontents
\section{Introduction}
\label{sec:introduction}
A $C^{\infty}$ function $f:M\to\mathbb{R}$ is said to be a Morse function if every critical point is non-degenerate and for every $u$ in $\mathbb{R}$, $f^{-1}[-\infty,u]$ is compact. A critical point is non-degenerate iff the Hessian matrix of $f$ at $c$ denoted as $H^f_c$ is non-singular. Let $Crit\ f$ denote the set of critical points of $f$. The $index$ of a non-degenerate critical point  $c\in Crit\ f$  is the number of negative eigenvalues of $H^f_{c}$. Let $Crt_\lambda f$ denote the set of critical points of $f$ of index $\lambda$. For a critical point $c$, in an arbitrary small neighbourhood of $c$, it is well known that, for appropriate coordinates $(u_1, . . . , u_n)$, $f$ is of the form $f(u_1, . . . , u_n) =f(c)-\sum_{j=1}^{\lambda} u_j^{2}+\sum_{j=\lambda+1}^{n} u_j^{2}$. 

For an arbitrary $r\in \mathbb{R}$, let
\[M^r_f:= f^{-1}(-\infty,r]=\{x\in M\mid f(x)\leq r\}\] 
We heavily use the fact that the topology of $M^r_f=f^{-1}(-\infty,r],-\infty<r<\infty$ changes only at critical points. Where, at each critical point of index $\lambda$, we attach a cell of dimension $\lambda$. The relation between the topology of $M$ and the Morse function is the topic of ~\cite{milnor2016morse}. We discussed the main results of Morse theory in section \ref{sec: prelim}.

Let $X\subset \mathbb{P}_N$ be a smooth algebraic subvariety in the complex projective space $\mathbb{P_N}$. Let $A\subset \mathbb{P}_N$ be a codimension two linear subspace ($\cong \mathbb{P}_{N-2}$) that intersect $X$ transversally. Then the pencil of hyperplane sections of $X$ with axis $X'$ is the map $P:Y\to \mathbb{P}_1$ where $Y$ is the blow up of $X$ along $X'$ and $\mathbb{P}_1$ is the space of hyperplanes containing $A$.  A pencil is called Lefschetz pencil if every fibre is either smooth or has at most one quadratic singularity.    

Lamotke~\cite{lamotke1981topology} proves Lefschetz hyperplane section theorem by using monodromy, thimbles and vanishing cycles of Lefschetz pencil, whereas we do it by composing the pencil with a Morse function on the complex projective line $\mathbb{P}_1$.

The composite function will be Morse in the case of plane curves. However, we prove that it is enough to study 
\[h:\mathbb{A}^2\to\mathbb{A}^2\]\[(x,z)\to(x,F(x,z))\] restricted to $C$, where $\mathbb{A}^2$ is the affine plane. We study the Jacobian of $h$, denoted as $J(h)$ and use Bézout's theorem (that states: two planes curves without common components will intersect in the product of degrees number points) to prove Theorem \ref{genusformulathm} in Section \ref{sec:extopo}.

\begin{theorem}[Genus Formula]
\label{genusformulathm}
Let $g(C)$ denote the topological genus of a curve $C$ of degree $d$. Then\[g(C)=\frac{1}{2}(d-1)(d-2)\]
\end{theorem}

Now, Let $C$ and $B$ are both compact Riemann surfaces, and $f: C\to B$ be a non-constant holomorphic map. Then $f$ is called a \emph{ramified covering map} of $B$ and $C$ is the \emph{ramified cover}.  There exist a finite set of points of $B$, outside of which we do find an honest covering. These exceptional points on $C$ are called the \emph{ramification points} and its image under $f$ are called \emph{discriminant points}. An \emph{unramified} covering or simply, covering is then the occurrence of an empty ramification locus. And $f$ is locally(around a point $p$) given by $z\mapsto z^{n_p}, n_p\geq 1$. If $n_p>1$, the point is called a \emph{ramified point} with \emph{ramification index} $n_p$. We prove the following theorem in Section \ref{sec:Riemann}.

\begin{theorem}[Riemann Hurwitz Theorem]
\label{RHT}Let $C$ and $B$ be two compact connected Riemann surfaces $f:C\to B$ be a non-constant holomorphic map of degree $d$. The genus of $C$, denoted as $g(C)$ is given by \begin{equation*}g(C)-1=d(g(B)-1)+\frac{1}{2}\sum_{p\in C}(n_p-1)
\end{equation*}where $g(B)$ is the genus of $B$ and $n_p$ is the ramification index of a point $p$. 
\end{theorem}

However, for a higher dimensional manifold, the composite function will have very bad singularities at the preimages of critical points $c^0$ of index $0$ and $c^2$ of index $2$ of $\mathbb{P}_1$. Studying the topology of complex varieties involve three steps. In the first step, we find the index of critical points of $Y$ coming from the Pencil. Followed by understanding the topology of the blow-up variety $Y$. The final step is comparing the topology of $Y$ with that of $X$. 

Let $D_{c^0}$ be a closed real-dimension 2 disc with boundary around $c^0$ that does not contain any critical values of the Pencil and $D_{c^2}$ be a closed real-dimension 2 disc with boundary around $c^0$ that does not contain any critical values of the Pencil. Let $B_0$, $B_2$ be the preimages of $D_{c^0}$, $D_{c^2}$ respectively. We use Morse theory to its full strength to find the relative homology $H_\lambda(Y\symbol{92}\mathring{B_2},B_0)$, where $\mathring{B_2}$ denotes the interior of $B_2$.
Further considering appropriate homology long exact sequences, we prove the following theorem in Section \ref{sec:LHSFT}.

\begin{theorem}[Lefschetz Hyperplane Section Theorem]\label{LHST} The inclusion $X_{c^0}\hookrightarrow X$ induces isomorphisms of the homology groups $H_\lambda(X_{c^0})\to H_\lambda(X)$ for every $\lambda < n-1$ and a surjection for $\lambda =n-1$.
\end{theorem}

\subsection*{Acknowledgement:} This work came out of first author's Masters Thesis and is supported partly by Infosys SHE scholarship. The authors would like to thank Amit Hogadi, Krishna Kaipa and Steven Spallone for the encouragement and helpful comments. 
\section{Preliminaries}\label{sec: prelim}
In this section, we state results regarding Morse Theory and Lefschetz Pencil necessary for proving Theorems \ref{genusformulathm}, \ref{RHT} and \ref{LHST}.

Stated below is the main theorem of Morse Theory that describes the role of critical points in general. The proof of it follows from homotopy theory and results about topological surgery. See~\cite{milnor2016morse} and~\cite{audin2014morse}.
\begin{theorem}
\label{morsehomology}
Let $M^r_f:= f^{-1}(-\infty,r]=\{x\in M\mid f(x)\leq r\}$. Suppose there exists $s$ in $Img\  f$ such that $s>r$ and $f^{-1}[r,s]$ is compact,
\begin{enumerate}
\item if $f^{-1}[r,s]$ does not contain any critical point, then $M^s_f$ is homotopic to $M^r_f$. Moreover, $M^{s}_f$ is a deformation retract of $M^{r}_f$. 
\item If $f^{-1}[r, s]$ contains a critical point $c^\lambda_i$ and let $f(c^\lambda_i)=r^\lambda_i$ be the corresponding critical value  then $M^{s}_f$ is homotopic to $M^{r}_f$ attached with a $\lambda$- cell.
\item If  $M^{r^\lambda_i}_f$ is compact $\forall c^\lambda_i \in Crit\ f $ then M has the homotopy type of a C-W complex with one cell of dimension $\lambda$ attached at every critical point $c^\lambda_i$.
\end{enumerate}
\end{theorem}
From Theorem \ref{morsehomology}, we know that each critical point of index $\lambda$ corresponds to a $\lambda$ cell in the C-W decomposition of $M$. Therefore we make the natural definition,
\[C_\lambda(M):=\langle Crit_\lambda f \rangle\]  where $R$ is an arbitrary ring and $\langle Crit_\lambda f \rangle$ is the free abelian group generated by the set $Crit_\lambda f$. We define $\partial_\lambda$ with respect to the moduli space between the critical points.

For all $ c^\lambda_i\in Crit_\lambda f$ we have
  \[\partial_\lambda:C_\lambda\to C_{\lambda-1}\]
 \vspace{0.25 cm}
 \[ \partial_\lambda(c^\lambda_i)= \sum_{ c^{\lambda-1}_j \in\   Crit_{\lambda-1}f}\# M_f(c^\lambda_i,c^{\lambda-1}_j)c^{\lambda-1}_j\] where $\# M_f(c^\lambda_i,c^{\lambda-1}_j)$ is the signed number of gradient flow lines  between $c^\lambda_i$ and $c^{\lambda-1}_j$. The Morse Theory complex is defined as
\begin{equation}
\label{homologycomplex}
\begin{tikzcd}
0 \arrow[r, "0"] & C_{n}(f,R) \arrow[r, "\partial_n" ] &\cdots \arrow[r]& C_{\lambda}(f,R) \arrow[r, "\partial_\lambda" ]
&\cdots \arrow[r] & C_{0}(f,R) \arrow[r, "\partial_0" ] & 0
\end{tikzcd}
\end{equation}
 This complex is equivalent to the homology complex associated with the C-W Complex because each index $\lambda$ critical point corresponds to a $\lambda$ cell.
\begin{lemma}\label{blemma}
On an $n$ dimensional manifold $M$, the Euler characteristic of $M$ denoted as $e(M)$ is the alternating sum of the number of index $\lambda$- cells for each $0\leq\lambda\leq n$.~\cite{milnor2016morse}
\end{lemma}
\begin{lemma}\label{alemma}Given a Morse function $f:M\to \mathbb{R}$, there exist a C-W Complex homotopic to $M$,  and the $\lambda$ cells are critical points of index $\lambda$.
\end{lemma}
For a proof see~\cite{milnor2016morse}.

\begin{proposition}
Let $M$ be a compact, connected and oriented  manifold and $f:M\to \mathbb{R}$ be a Morse function. Let $p : X\to M$ be a finite connected topological cover, then $f':X\to \mathbb{R}$ given by $f'=f\circ p$ is a Morse function and $e(X)=de(M).$
\end{proposition}
\begin{proof}A finite connected cover of a compact connected manifold of dimension $n$ is a compact connected topological manifold of dimension $n$. Therfore, $X$ admits a Morse decomposition. Since $p$ is a local isomorphism, critical points of $f'$  of index $\lambda$ are precisely $p^{-1}(Crit\ f)$. We appropriately call $p^{-1}(c_i^\lambda)$ as $a^\lambda_{i_1}, a^\lambda_{i_2},\cdots a^\lambda_{i_d}$. Since there are $d$ copies of each cell, by Lemma \ref{blemma}, $e(X)=de(M)$.
\end{proof}
Let $\mathbb{P}_N$ denote $N$-dimensional complex projective space.  
\begin{defn}Let $X\subset \mathbb{P}_N$ be a smooth irreducible variety.
A \emph{pencil} in $\mathbb{P}_N$
consists of all hyperplanes which contain a fixed $(N-2)$-dimensional projective subspace
$A$, which is called the \emph{axis} of the pencil. This axis is uniquely determined by intersecting any two hyperplanes of the Pencil.
\end{defn}

\begin{defn}Let $X\subset\mathbb{P}_N$ be a closed irreducible variety of dimension $n$. The points of tangential intersection of the pencil $P$ with $X$ are called singular points of the pencil. Such a pencil that only admits finitely many tangential intersections with $X$ and the singular points are quadratic is called a \emph{Lefschetz Pencil}.
\end{defn}
Define \[X_t:=X\cap H_t,\ t\in \mathbb{P}_1\]
such that \[X=\cup_{t\in \mathbb{P}_1}X_t.\]
Let \[X'=X\cap A\] and \[Y=\{ (x, t) \in X\times \mathbb{P}_1 \mid x \in H_t\}.\] The map $p:Y\to X$ is called the blowup map and $Y$ is called the \emph{blowup variety} or \emph{modification} of $X$ along $X'$. Let \begin{equation}\label{eq:1}Y':=p^{-1}(X')=X'\times \mathbb{P}_1\end{equation} and $Y_t:=P^{-1}(t)$. 
Then $P\mid_{Y_t}:Y_t\to X_t$ is a biregular map. Therefore $Y_t$ is isomorphic to $X_t$.
\begin{defn}
The space of all hyperplanes of $\mathbb{P}_N$ which are tangent to $X$ denoted as $\check{X} \subsetneq \check{\mathbb{P}}_N$ is called the \emph{dual variety} of $X$ and is a proper closed subvariety of $\check{P}^N$. 
\end{defn}
\section{Genus Formula for Plane Curves}
\label{sec:extopo}
Equipped with Morse theory, we study the topology of projective varieties, the simplest being curves. 
In this section, we prove Theorem \ref{genusformulathm}

Let $f$ be a degree $d$ irreducible polynomial in homogenous coordinates $(x:y:z)$. We define $C:=V(f)$, the curve defined by $f$ and assume that $C$ is smooth.

\begin{lemma}The Axis of a Pencil in $\mathbb{P}_2$ is a point. We choose a point $p$ as the axis such that $p\notin C$. Let $P:C\to \mathbb{P}_1$ be the Lefschetz pencil whose axis $A$ is $p\notin C$. We choose coordinates $(x:y:z)$ such that $p=(0:0:1)$. Let the projection from the Axis be given by \begin{align*}\pi_p:\mathbb{P}_2\symbol{92}\{p\}&\to \mathbb{P}_1\\(x:y:z)&\mapsto(x:y).\end{align*}\begin{enumerate}
\item $\pi_p$ is well defined 
\item the fiber of $\pi_p$ is the same as the fiber of $P$. Therefore Lefschetz Pencil on the Curve $C$ is exactly the projection map $\pi_p|_C$
\end{enumerate}
\end{lemma}
\begin{proof}
For every $(x:y:z)\neq (0:0:1)$, $\pi_p(x:y:z)=(x:y)\neq (0,0)$. Therefore $\pi_p(x:y:z)=(x:y)\in \mathbb{P}_1$.

For a point $(x_0:y_0)\in \mathbb{P}_1$ the fiber $\pi_p^{-1}(x_0:y_0)=(x_0:y_0:z)=(x_0:y_0:0)+z(0:0:1)$. ie, any point on the subspace $(x_0:y_0:z)$ is given by $\lambda_1(x_0:y_0:0)+\lambda_2z(0:0:1)$ where $\lambda_1,\lambda_2 \in \mathbb{C}$. Therefore $(x_0:y_0:z)$ is the line passing through $(0:0:1)$. Therefore $\pi_p$ corresponds to the Lefschetz pencil whose axis is $(0:0:1)$.
\end{proof}
For any two points $c^0$ and $c^2$ on $\mathbb{P}_1$, by Lemma \ref{alemma} and Reeb's Theorem~\cite{milnor2016morse}[]Theorem 4.1], there exists a Morse function such that
$c^0$ is its index $0$ critical point, $c^1$ is its index $2$ critical point and there are no index $1$ critical points.
\begin{proposition}
Consider the map
\[C\stackrel{\mathrm{{\pi_p|_C}}}{\longrightarrow} {\mathbb{P}_1}\stackrel{\mathrm{\emph{g}}}{\longrightarrow}{\mathbb{R}}\]
where $g$ is a morse function on $\mathbb{P}_1$ with exactly $2$ critical points ,one of index $0$ and other of index $2$ such that the index 0 and index 2 critical points of $g$ are regular values of $\pi_p$. The composite map $(g\circ \pi_p)|_C$ is Morse.
\end{proposition}
\begin{proof}
The curve $C$ is smooth, therefore locally isomorphic to 1-dimensional Disc. Consider such a disc $D$ around a critical point and let $z=x+iy\in D$. $\pi_p|_D:D\to D$ is given by \begin{equation}z\longrightarrow z^2\end{equation}\begin{equation}x+iy\longrightarrow x^2-y^2+i2xy\end{equation}

Since a Morse function $g$ on $\mathbb{P}_1$ with exactly 2 critical points say $c^0$ and $c^2$, and the Jacobian  has maximal rank on $\mathbb{P}_1\symbol{92}\{c^0\cup c^2\}$ and therefore is of the form \begin{equation}
\label{g}g(X,Y)=aX+bY +\ \text{higher order terms},\  a\neq 0\ or\ b\neq 0\end{equation}Since for a $z\in C$, $\partial (g\circ\pi_p)(z)=\partial g(\pi_p(z))\partial \pi_p(z)$, 
\begin{equation}
\label{critcomp}Crit\ g\circ\pi_p= \pi_p^{-1}(Crit\ g)\cup Crit\ \pi_p\end{equation}

Let us call the index $0$ and index $2$ critical points of $g$ as $c^0$ and $c^2$ respectively. Since $\pi_p$ is a local isomorphism in the neighborhood of a critical point of $g$, $\pi_p^{-1}(c^0)$ are non-degenerate, index $0$ critical points of $(g\circ\pi_p)|_C$. Similarly, $\pi_p^{-1}(c^2)$ are non-degenerate, index $2$ critical points of $(g\circ\pi_p)|_C$. Therefore it is enough to prove that critical point of $\pi_p$ are non-degenerate index 1 critical points of $(g\circ\pi_p)$, i.e, for every $z=x+iy\in Crit\ \pi_p$ the Hessian matrix of $\pi_p$ is non-singular.
\begin{equation}(g\circ \pi_p)(z)=a(x^2-y^2)+2bxy +\ higher\ order\ terms\end{equation}
\[H^{(g\circ \pi_p)}(z)(0)=\begin{pmatrix}
  2a & 2b\\ 
  2b &   -2a\\
\end{pmatrix}\]
\[det(H^{(g\circ \pi_p)}(x+iy)(0))=det\begin{pmatrix}
  2a & 2b\\ 
  
   2b &   -2a\\
\end{pmatrix}=-4a^2-4b^2\]\[-4(a^2+b^2)=0\Leftrightarrow a=0\ and\ b=0\]
Since either $a\neq 0$ or $b\neq 0$,we have that the Hessian matrix of $\pi_p$ is non-singular and critical points of $g\circ \pi_p$ are non-degenerate.
\end{proof}

Consider $\mathbb{P}^2\symbol{92}\{(0:0:1)\}$. This space is defined by $(x:y:z)$ such that $x\neq0$ or $y\neq 0$.  Take $y=1$ and \[(x:1:z)\cong (x:z)\]
\begin{rem}Let $\pi_x:\mathbb{A}^2\to\mathbb{A}^1$ be the projection $(x:z)\to x$. This is exactly $\pi_p|_{y=1}$. Therefore one can easily see that our familiar coordinate projection is a special case of $\pi_p$.\end{rem}

We prove that critical points of $\pi_p$ are exactly index $1$ critical points. This can be easily seen since $\pi_p$ is a local isomorphism and the preimage of local maxima/minima of $g$ would be local maxima/minima of $\pi_p\circ g$. Another proof is by explicitly calculating the eigenvalues.\\The characteristic polynomial of $H^{g\circ \pi_p}(x+iy)$ is given by $(X-2a)(X+2a)-(-2b)^2$.
\[(X-2a)(X+2a)-(-2b)^2=X^2-4a^2-4b^2\]\begin{equation}\label{CritP}X=\pm\sqrt{4a^2+4b^2}\end{equation}
Therefore the number negative eigenvalues of $H^{g\circ \pi_p}$ is $1$.

Define $F(x:z):=f(x:1:z)$. It is seen that degree of $F$ is equal to $d$. Otherwise $y$ would be a factor of $f$ in which case $f$ is not irreducible.
Consider the map \[h:\mathbb{A}^2\to\mathbb{A}^2\]\[(x,z)\to(x,F(x,z))\]
\begin{proposition}
\label{C}
The critical points of $\pi_p|_{C\cap{\mathbb{A}^2}}$ are exactly those of $h$ when restricted to $F(x,z)=0$.
\end{proposition}
\begin{proof}
\[h|_C:(x,z)\longrightarrow (x,0)\]
\[(x_0,z)\longleftarrow x_0\]
\[\pi_p|_{C\cap{\mathbb{A}^2}}:(x:1:z)|_{\mathbb{A}^2}\longrightarrow(x:1)\]
\[(x_0,z)\longleftarrow x_0\]
The fibers of both the maps are the same.
\end{proof}
\begin{lemma}\label{Pd}Lefschetz pencil $P$ has exactly $d(d-1)$ critical points.\end{lemma}
\begin{proof}
By Proposition \ref{C}, it is enough to calculate the number of critical points of $h|_{F(x:z)=0}$. The Jacobian of $h$ is give by
\[J(h)=\begin{pmatrix}
  \partial x/\partial x & \partial F/\partial x\\ 
  
   \partial x/\partial z &   \partial F/\partial z\\
\end{pmatrix}=\begin{pmatrix}
  1& \partial F/\partial x\\ 
  
   0 &   \partial F/\partial z\\
\end{pmatrix}\] 
\[det(J(h))=det \begin{pmatrix}
  1& \partial F/\partial x\\ 
  
   0 &   \partial F/\partial z\\
\end{pmatrix}=0\Leftrightarrow \partial F/\partial z=0\]
Therefore, the critical points of $h|_C$ is given by the intersection of the two curves $f=0$ and $\partial F/\partial z=0$.
Now by Bézout's theorem, the intersection has $d(d-1)$ points.
\end{proof}
\begin{corollary}
The degree of the dual curve is the intersection number of the dual curve with a Projective line. This translates to the number of lines in the Lefschetz pencil that is tangent to $X$. That is, $Class \ C= degree\ of\ \check{C}$ is equal to $d(d-1)$.\end{corollary}

\begin{proof}[\textbf{Proof of Theorem \ref{genusformulathm}}]
It is easily seen that the number of index 0 critical pints is the cardinality of $\pi_p^{-1}(0)$ is equal to the cardinality of $\mathbb{P}_1\cap C$ is equal to $d$. Similarly, for index $2$ critical points. Index one critical points come from $\pi_p$ and by the previous proposition, there is exactly $d(d-1)$ critical points of index $1$.

Therefore, the Homology complex is given by
\[
\begin{tikzcd}
0 \arrow[r, "0"] & C_{2} \arrow[r, "\partial_2" ]\arrow[d, white, "\ovl" black]  \arrow[r]& C_{1} \arrow[r, "\partial_1" ] \arrow[d, white, "\ovl" black] & C_{0} \arrow[r, "\partial_0" ]\arrow[d, white, "\ovl" black] & 0\\
&\mathbb{Z}^d \arrow[transparent, "\partial_n" ] & \mathbb{Z}^{d(d-1)} \arrow[transparent, "\partial_n" ] \arrow[transparent]& \mathbb{Z}^d \arrow[transparent, "\partial_0" ] & 
\end{tikzcd}
\]
\begin{flushleft}
Taking as a matter of fact, that for a connected, compact orientable surface, both $H_0$ and $H_2$ are isomorphic to $\mathbb{Z}$ the following calculation,
\end{flushleft}

\[H_0 \cong \frac{ Ker \ \partial_0}{Im\ \partial_1 } \cong \mathbb{Z} \implies \frac{\mathbb{Z}^d}{Im\ \partial_1 }\cong \mathbb{Z} \implies Im\ \partial_1 \cong \mathbb{Z}^{d-1} \implies Ker \ \partial_1 \cong \mathbb{Z}^{d(d-1)-(d-1)}\]

\[H_2 \cong  Ker \ \partial_2 \cong \mathbb{Z} \implies Im\ \partial_2 \cong \mathbb{Z}^{d-1}\]
\[H_1 \cong \frac{ Ker \ \partial_1}{Im\ \partial_2 } \cong {\mathbb{Z}^{d(d-1)-2(d-1)}}\cong {\mathbb{Z}^{(d-2)(d-1)}}\]

\begin{flushleft}
yields the result , genus of $C$ is given by $g(C)=\frac{1}{2}(d-1)(d-2)$.
\end{flushleft}
\end{proof}

\section{Riemann-Hurwitz Theorem}
\label{sec:Riemann}
In this section, we prove Theorem \ref{RHT}.

The Riemann-Hurwitz Theorem, named after Bernhard Riemann and Adolf Hurwitz, describes the relationship of the genera of two compact Riemann surfaces when one is a ramified covering of the other.  We use deformation Theory explicitly and construct a deformation function with our required properties to fit into the setup of Morse theory. 
\begin{defn}
For a map $f$, where $p$ is any point in $C$, one can choose an isomorphism $\phi_C$ from the open unit disc $D$ to the neighborhood of $p$ and another isomophism $\phi_B$ from $D$ to a neighborhood of $f(p)$ and by open mapping theorem $\phi_B^{-1}\circ f\circ \phi_C \mid_D:D\to D$ is given by $z\mapsto z^{n_p}, n_p\geq 1$, where $n_p$ is the order of vanishing of $f$ at $p$ for $z$ in a neighborhood of $p$. If $n_p=1$, the point is called an \emph{unramified point} of $f$. If $n_p>1$, the point is called a \emph{ramified point} with \emph{ramification index} $n_p$. The \emph{degree} of $f$ is the cardinality of any regular fibre. 
\end{defn}
With abuse of notation, 
we write $f:D\to D$ such that the local coordinates are chosen to satisfy $p=f(p)=0$.

An equivalent way of thinking about ramified point is that $p$ is said to be ramified with respect to $f$ if there exists a small neighborhood $U$ of $p$ such that $f(p)$ has exactly one preimage in $U$, but the image of any other point in $U$ has exactly $n_p$ preimages in $U$.

\begin{proof}[\textbf{Proof of \ref{RHT}}]
 We first prove case where $f$ has only non-degenerate critical points, i.e, $f:D\to D$ is given by $z\mapsto z^{n_p}$ where $n_p\leq2$ for all $p\in C$. 
 
 Let $g$ be a Morse function on $B$ such that critical points of $g$ are regular values of $f$.
Consider the composite map
\[C\stackrel{\mathrm{{f}}}{\longrightarrow} {B}\stackrel{\mathrm{\emph{g}}}{\longrightarrow}{\mathbb{R}}.\]

 Let $l_\lambda^B$, $l_\lambda^C$ denote the number of index $\lambda$-critical points in $B$ and $C$ respectively, and $\mu$ the number of critical points of $f$ then from (\ref{critcomp}) and (\ref{CritP}), we have that
\begin{equation}
\label{ls}
\begin{split}
l_0^C&=d\times l_0^B\\
l_2^C&=d\times l_2^B\\
l_1^C&=d\times l_1^B+\mu
\end{split}
\end{equation}
 Substitituting (\ref{ls}) in $e(C)=l_0^C-l_1^C+l_2^C$,
\begin{equation*}\begin{split}
 e(C)&=(d\times l_0^B)-(d\times l_1^B+\mu)+(d\times l_2^B)\\
 &=d(l_0^B-l_1^B+l_2^B)-\mu   
\end{split}
\end{equation*}

That is,
\[
e(C)=de(B)-\mu
\]

Let $g$ denote the genus of the curve. Substituting $e=2-2g$, we get
\begin{equation}
\label{genus}
g(C)-1=d(g(B)-1)+\frac{1}{2}\mu  \end{equation}

We can quickly verify our genus formula for smooth plane curves of degree $d$ as in Section \ref{sec:extopo} by replacing $B$ in (\ref{genus}) with $\mathbb{P}_1$. It follows from Lemma \ref{Pd} that $\mu=d(d-1)$, where $d$ is the degree of the curve. Therefore,
\begin{equation}
\begin{split}
g(C)-1&=d(g(B)-1)+\frac{1}{2}d(d-1)\\
g(C)&=\frac{1}{2}(d^2-3d+2)\\
&=\frac{1}{2}(d-1)(d-2)
\end{split}
\end{equation}
as desired.

Now, we look at the case where $C$ has degenerate critical points, i.e, $f:D\to D$ is given by $z\mapsto z^{n_p}$, where $n_p>2$. Inorder to make the critical points non-degenerate, we deform the function locally by a small perturbation in the unit disc $D$ centered at $p$.  \[F:D_\frac{1}{2}\times D \to D\times D\] given by 
 \[F(z,t):=(f_t(z),t)\]
  where $D_\frac{1}{2}$ is a disc of radius $\frac{1}{2}$ contained in $D$ and $f_t$ is locally given by
  \[f_t:D_\frac{1}{2}\to D\]
  \begin{equation}
  \label{f}f_t(z):=z^{n_p}-tz.\end{equation}
 \[J^{F}(z,t)=\begin{pmatrix}
  \partial f_t/\partial z & \partial t/\partial z\\ 
  
   \partial f_t/\partial t & \partial t/\partial t\\
\end{pmatrix}=\begin{pmatrix}
  n_pz^{n_p-1}-t& 0\\ 
  
   -z &   1\\
\end{pmatrix}\] 
\[det(J^{F_p})=n_pz^{n_p-1}-t=0\Rightarrow z=\sqrt[\leftroot{-3}\uproot{3}{n_p-1}]{\tfrac{t}{n_p}}\]
Therefore, for a fixed $\epsilon <<\frac{1}{2}$ we choose a $t$ such that $\mid t\mid<n_p\epsilon^{n_p-1}$. Then  $f_{t}(z):D_\frac{1}{2}\to D$ has
$n_p-1$ distinct critical points contained in $D_\epsilon$. And the Hessian of $f_t$ given by \[\begin{split}H^{f_t}(z)&=\frac{\partial^2 f_t}{\partial z^2}\\
&=n_p(n_p-1)z^{n_p-2}\\
&=0\Leftrightarrow z=0
\end{split}\]
But $z=0$ is not a critical point of $f_t$.
Therefore, all the $n_p-1$ critical points of $f_p$ are non-degenerate and clearly contained in $D_\epsilon$.

For the triple $D$, $D_{\epsilon}$, $C\symbol{92}\overline{D_{\frac{1}{2}}}$,
by partitions of unity theorem, there exists a $\mathcal{C}^\infty$ function $\phi:C\to[0,1]$ such that $\phi$ is identically $1$ on $D_{\epsilon}$, $0$ on $C\symbol{92}\overline{D_{\frac{1}{2}}}$ and $0\leq\phi\leq1$ on $\overline{D_{\frac{1}{2}}\symbol{92}D_{\epsilon}}$, the closure of $D_{\frac{1}{2}}\symbol{92}D_{\epsilon}$. Here the closure of $D_{\frac{1}{2}}$, $\overline{D_{\frac{1}{2}}}$ is the compact support of $\phi$.
Recall that $f-f_t=tz$ from (\ref{f}). 

We define \begin{equation}
\label{fone}
\begin{split}f_1&=(1-\phi)f+\phi f_t\\
&=f-\phi(f-f_t)\\
&=f-tz\phi
\end{split}\end{equation}

Since $f_1=f$ on $C\symbol{92}\overline{D_{\frac{1}{2}}}$. Therefore the critical points of $f_1$ are the same as the critical points of $f$ on $C\symbol{92}\overline{D_{\frac{1}{2}}}$. Again, $f_1=f_t$ on $D_\epsilon$. Therefore, the critical points of $f_1$ are the same as the critical points of $f_t$ on $D_\epsilon$. 

By applying the following lemma to $\psi=z\phi$, gives us that $f_1$ does not have any critical point in the region $\overline{D_{\frac{1}{2}}\symbol{92}D_\epsilon}$.

\begin{lemma}
Suppose $f:D\to \mathbb{C}$ be differentiable map that is a local isomorphism (i.e, $f$ has maximal rank everywhere) on $\overline{D_{\frac{1}{2}}\symbol{92}D_{\epsilon}}$. Let $\psi$ be a continuously differentiable function on $D$, then $f+t\psi$ is also a local isomorphism  on  $\overline{D_{\frac{1}{2}}\symbol{92}D_{\epsilon}}$ for sufficiently small $t>0$. 
\end{lemma}
\begin{proof}
Consider the map\[\Phi: \overline{D_{\frac{1}{2}}\symbol{92}D_{\epsilon}}\times [0,1]\to M_2(\mathbb{R})\]
\[(z,t)\mapsto J(f+t\psi).\]
Then $\Phi$ is continuous. As $GL_2(\mathbb{R})$ is open in $M_2(\mathbb{R})$,  $\Phi^{-1}(GL_2(\mathbb{R}))$ is open in $\overline{D_{\frac{1}{2}}\symbol{92}D_{\epsilon}}\times [0, 1]$ in the product topology.

Since $J(f)$ has maximal rank, $\overline{D_{\frac{1}{2}}\symbol{92}D_{\epsilon}}\times \{0\}\in \Phi^{-1}(GL_2(\mathbb{R}))$. Now, as $\overline{D_{\frac{1}{2}}\symbol{92}D_{\epsilon}}$ is compact, there exists an open set $\overline{D_{\frac{1}{2}}\symbol{92}D_{\epsilon}}\times [0, t_0)\supset \overline{D_{\frac{1}{2}}\symbol{92}D_{\epsilon}}\times \{0\}$ in $\Phi^{-1}(GL_2(\mathbb{R}))$. Therefore, for every $t<t_0$, $J(f+t\psi)$ has maximal rank.\end{proof}
This leads us to the following theorem.
\begin{theorem}Let $C$ and $B$ be compact connected Riemann surfaces.
 Let $f:C\to B$ be a $\mathcal{C}^\infty$ function that is analytic at every critical point. Assume that $f$ has $n$ degenerate critical points and $n_1$ non-degenerate critical points. Let $p$ be a non degenerate critical point of $f$. Then, there exists a $\mathcal{C}^\infty$ function $f_1:C\to B$ which has $n_p-1$ analytic non-degenerate critical points in a small neighborhood of $p$ and coincides with $f$ outside the neighborhood. Furthermore, there exists a $\mathcal{C}^\infty$ function $f_n:C\to B$ which has $0$ degenerate critical points and $n_1 + \sum_{p\in C}n_p-1$ analytic non-degenerate critical points.
 \end{theorem}
 \begin{proof}
  The first part of the proof follows from our explicit construction of the $\mathcal{C}^\infty$ function $f_1$ in (\ref{fone}) and subsequent verifications. Now we apply this to $f_1$. Suppose $q$ is a degenerate critical point of $f_1$ then,  there exists a $\mathcal{C}^\infty$ function $f_2$ which has $n-2$ degenerate critical points and $n_1+(n_p-1)+(n_q-1)$ non degenerate critical points.
  Inductively construct $f_3,\cdots ,f_n$. Note that $f_n$ has 0 degenerate critical points and \[n_1 +\sum_{p\ \text{degenerate}}(n_p-1)\] non degenerate index $1$ critical points.
  Since $n_p=2$ for a non degenerate critical point,
\begin{equation}
\sum_{p\ \text{non-degenerate}}(n_p-1)=\sum_{p\ \text{non-degenerate}}1=n_1.\end{equation}
 Also note that for a regular point $n_p=1$ and $\sum_{p\ \text{regular}}(n_p-1)=0$.
 Therefore, $f_n$ has 
 \[n_1 +\sum_{p\ \text{degenerate}}(n_p-1)=\sum_{p\in C} (n_p-1)\] non-degenerate critical points.
\end{proof}
\textbf\emph{{Proof of Proposition (\ref{RHT}) continued.}} Let $D$ be the neighborhood of a critical point of $f_n.$
   The function $f_n\mid_D:D\to D$ is given by $z\mapsto z^2$. Choose a Morse function $g:B\to \mathbb{R}$ such that critical points of $g$ are regular values of $f_n$. The composite map $g\circ f_n\mid_D:D\to \mathbb{R}$ is holomorphic and has non degenerate critical points. From (\ref{CritP}) the non degenerate critical points of $f_n$ have index 1. If $l_\lambda^B$, $l_\lambda^C$ denote the number of index $\lambda$-critical points in $B$ and $C$ of $f$ and $g\circ f$ respectively, then from (\ref{critcomp}),  
\begin{equation}\begin{split}
\label{lsfinal}l_0^C&=d\times l_0^B\\
l_2^C&=d\times l_2^B\\
l_1^C&=d\times l_1^B+\sum_p(n_p-1)
\end{split}
\end{equation}
Substituting (\ref{lsfinal}) to $e(C)=l_0^C-l_1^C+l_2^C$, 
\begin{equation*}\begin{split}
 e(C)&=(d\times l_0^B)-(d\times l_1^B+\sum_p(n_p-1))+(d\times l_2^B)\\
 &=d(l_0^B-l_1^B+l_2^B)-\sum_p(n_p-1)   
\end{split}
\end{equation*}
That is,
\begin{equation}
e(C)=de(B)-\sum_p(n_p-1)
\end{equation}

Let $g$ denote the genus of the curve. Substituting $e=2-2g$ in (\ref{genus}), we get
\begin{equation}g(C)-1=d(g(B)-1)+\frac{1}{2}\sum_{p\in C}(n_p-1)\end{equation}
\end{proof}

\section{Lefschetz Hyperplane Section Theorem}
\label{sec:LHSFT}
In the previous two sections, we completed the study of pencils in dimension 1 and in fact, generalized to any map between curves. Here we study the topology of complex projective varieties, give an illustration of the maps and prove Theorem \ref{LHST}

\subsection{Non Degeneracy of Critical Points of the Pencil}
Let $X$ be a smooth variety of complex dimension $n>1$ and $Y$ it's blow up along the smooth axis $X'$. Let $g$ be a morse function on $\mathbb{P}_1$ with exactly 1 index-$0$ and 1 index-$2$ critical points which are regular values of the Lefschetz pencil $P$ from $Y$ to $\mathbb{P}_1$. From (\ref{g}), $g$ is given by $g(x,y)=ax+by +\ \text{higher order terms}$, where $a\neq 0\ or\ b\neq 0$.
For a sufficiently small neighborhood $D$ of a critical point $c$, 
\[P\mid_D:D\to D\] is given by\[(z_1,z_2,...z_n)\mapsto\sum_{i=1}^n z_i^2\]
\[(x_1,iy_1,\cdots ,x_n,iy_n)\mapsto \sum_{i=1}^n x_i^2-y_i^2+i2x_iy_i=(\sum_{i=1}^n x_i^2-y_i^2,\sum_{i=1}^n2x_iy_i)\]
 Consider the composite map
\[Y\stackrel{\mathrm{{P}}}{\longrightarrow} {\mathbb{P}_1}\stackrel{\mathrm{\emph{g}}}{\longrightarrow}{\mathbb{R}}\]
given by
\[(g\circ P)\mid_D(x_1+iy_1,\cdots ,x_n +iy_n)=g(\sum_{i=1}^n x_i^2-y_i^2,\sum_{i=1}^n2x_iy_i)=a\sum_{i=1}^n(x_1^2-y_1^2)+2b\sum_{i=1}^n(x_1y_1)+O(x,y).\] 
%
Now, the Hessian matrix of $g\circ P\mid_D$ at a critical point is given by
\begin{align*}\begin{split}
\sbox0{$\begin{matrix}a& &0\\ &\ddots& \\0& &+a\end{matrix}$}
\sbox1{$\begin{matrix}+b& &0\\ &\ddots& \\0& &+b\end{matrix}$}
\sbox2{$\begin{matrix}-a& &0\\ &\ddots& \\0& &-a\end{matrix}$}
\sbox3{$\begin{matrix}0& &0\\ &\ddots& \\0& &0\end{matrix}$}
\sbox4{$\begin{matrix}\frac{-(a^2+b^2)}{a}& &0\\ &\ddots& \\0& &\frac{-(a^2+b^2)}{a}\end{matrix}$}
H^{(g\circ P)\mid_D}(0)&= 
\sbox0{$\begin{matrix}\frac{\partial^2 (g\circ P)}{\partial x_1^2}& \cdots& \frac{\partial^2 (g\circ P)}{\partial x_1x_n}\\ 
\vdots& \ddots& \vdots \\ \frac{\partial^2 (g\circ P)}{\partial x_nx_1}& \cdots& \frac{\partial^2 (g\circ P)}{\partial x_n^2}\end{matrix}$}
\sbox1{$\begin{matrix}\frac{\partial^2 (g\circ P)}{\partial x_1 \partial y_1}& \cdots& \frac{\partial^2 (g\circ P)}{\partial x_1\partial y_n}\\ 
\vdots& \ddots& \vdots \\ \frac{\partial^2 (g\circ P)}{\partial x_ny_1}& \cdots& \frac{\partial^2 (g\circ P)}{\partial x_n\partial y_n}\end{matrix}$}
\sbox2{$\begin{matrix}\frac{\partial^2 (g\circ P)}{\partial x_1^2}& \cdots& \frac{\partial^2 (g\circ P)}{\partial x_1x_n}\\ 
\vdots& \ddots& \vdots \\ \frac{\partial^2 (g\circ P)}{\partial x_nx_1}& \cdots& \frac{\partial^2 (g\circ P)}{\partial x_n^2}\end{matrix}$}
\sbox3{$\begin{matrix}\frac{\partial^2 (g\circ P)}{\partial y_1^2}& \cdots& \frac{\partial^2 (g\circ P)}{\partial y_1x_n}\\ 
\vdots& \ddots& \vdots \\ \frac{\partial^2 (g\circ P)}{\partial y_nx_1}& \cdots& \frac{\partial^2 (g\circ P)}{\partial y_n \partial x_n}\end{matrix}$}
\sbox4{$\begin{matrix}\frac{\partial^2 (g\circ P)}{\partial y_1^2}& \cdots& \frac{\partial^2 (g\circ P)}{\partial y_1y_n}\\ 
\vdots& \ddots& \vdots \\ \frac{\partial^2 (g\circ P)}{\partial y_ny_1}& \cdots& \frac{\partial^2 (g\circ P)}{\partial y_n^2}\end{matrix}$}
\left[
\begin{array}{c|c}
\usebox{0}&\usebox{1}\\
\hline
  \vphantom{\usebox{1}}\usebox{3}&\usebox{4}
\end{array}
\right]\\
&=
\sbox0{$\begin{matrix}+2a& &0\\ &\ddots& \\ 0& &+2a\end{matrix}$}
\sbox1{$\begin{matrix}+2b& &0\\  &\ddots& \\0& &+2b\end{matrix}$}
\sbox2{$\begin{matrix}-2a& &0\\ &\ddots& \\ 0& &-2a\end{matrix}$}
\sbox3{$\begin{matrix}0& &0\\ &\ddots& \\0& &0\end{matrix}$}
\sbox4{$\begin{matrix}\frac{-(a^2+b^2)}{a}& &0\\ &\ddots& \\0& &\frac{-(a^2+b^2)}{a}\end{matrix}$}
\left[
\begin{array}{c|c}
\usebox{0}&\usebox{1}\\
\hline
  \vphantom{\usebox{1}}\usebox{1}&\usebox{2}
\end{array}
\right]\end{split}\end{align*}
where each block is an $n\times n$ diagonal  matrix with 0 in every $ij$-th entry for $i\neq j$.
\begin{equation}
\sbox0{$\begin{matrix}+a& &0\\ &\ddots& \\0& &+a\end{matrix}$}
\sbox1{$\begin{matrix}+b& &0\\ &\ddots& \\0& &+b\end{matrix}$}
\sbox2{$\begin{matrix}-a& &0\\ &\ddots& \\0& &-a\end{matrix}$}
H^{(g\circ P)\mid_D}(0)=2\left[
\begin{array}{c|c}
\usebox{0}&\usebox{1}\\
\hline
  \vphantom{\usebox{1}}\usebox{1}&\usebox{2}
\end{array}
\right]
\end{equation}
Let $a\neq 0$.
We reduce the matrix to row echelon form using the the row transformation
\[R_{n+1}\mapsto R_{n+1}-\frac{b}{a}R_{1}.\]
\[\sbox0{$\begin{matrix}+a& &0\\ &\ddots& \\0& &+a\end{matrix}$}
\sbox1{$\begin{matrix}+b& &0\\ &\ddots& \\0& &+b\end{matrix}$}
\sbox2{$\begin{matrix}-a& &0\\ &\ddots& \\0& &-a\end{matrix}$}
\sbox3{$\begin{matrix}0& &0\\ &\ddots& \\0& &0\end{matrix}$}
\sbox4{$\begin{matrix}\frac{-(a^2+b^2)}{a}& &0\\ &\ddots& \\0& &\frac{-(a^2+b^2)}{a}\end{matrix}$}
\left[
\begin{array}{c|c}
\usebox{0}&\usebox{1}\\
\hline
  \vphantom{\usebox{1}}\usebox{1}&\usebox{2}
\end{array}
\right]\mapsto \left[
\begin{array}{c|c}
\usebox{0}&\usebox{1}\\
\hline
  \vphantom{\usebox{1}}\usebox{3}&\usebox{4}
\end{array}
\right]
\]
and the determinant of $H^{(g\circ P)\mid_D}$ at a critical point, given by
\begin{align}\begin{split}
\sbox0{$\begin{matrix}+a& &0\\ &\ddots& \\0& &+a\end{matrix}$}
\sbox1{$\begin{matrix}+b& &0\\ &\ddots& \\0& &+b\end{matrix}$}
\sbox2{$\begin{matrix}-a& &0\\ &\ddots& \\0& &-a\end{matrix}$}
\sbox3{$\begin{matrix}0& &0\\ &\ddots& \\0& &0\end{matrix}$}
\sbox4{$\begin{matrix}\frac{-(a^2+b^2)}{a}& &0\\ &\ddots& \\0& &\frac{-(a^2+b^2)}{a}\end{matrix}$}
Det\left[
\begin{array}{c|c}
\usebox{0}&\usebox{1}\\
\hline
  \vphantom{\usebox{1}}\usebox{1}&\usebox{2}
\end{array}
\right]&= 
\sbox0{$\begin{matrix}+a& &0\\ &\ddots& \\0& &+a\end{matrix}$}
\sbox1{$\begin{matrix}+b& &0\\ &\ddots& \\0& &+b\end{matrix}$}
\sbox2{$\begin{matrix}-a& &0\\ &\ddots& \\0& &-a\end{matrix}$}
\sbox3{$\begin{matrix}0& &0\\ &\ddots& \\0& &0\end{matrix}$}
\sbox4{$\begin{matrix}\frac{-(a^2+b^2)}{a}& &0\\ &\ddots& \\0& &\frac{-(a^2+b^2)}{a}\end{matrix}$}
Det \left[
\begin{array}{c|c}
\usebox{0}&\usebox{1}\\
\hline
  \vphantom{\usebox{1}}\usebox{3}&\usebox{4}
\end{array}
\right]\\
&=(-1)^n\frac{(a^2+b^2)^n}{a^n}a^n\\
&=(-1)^n(a^2+b^2)^n\\
&\neq0
\end{split}
\end{align}

Therefore, the critical points of $g\circ P$, which are the critical points of the pencil $P$ are nondegenerate.

The characteristic polynomial of 
$H^{(g\circ P)\mid_D}$ at a critical point is given by
\begin{align}\begin{split}
\sbox0{$\begin{matrix}a-x& &0\\ &\ddots& \\0& &a-x\end{matrix}$}
\sbox1{$\begin{matrix}+b& &0\\ &\ddots& \\0& &+b\end{matrix}$}
\sbox2{$\begin{matrix}-a-x& &0\\ &\ddots& \\0& &-a-x\end{matrix}$}
\sbox3{$\begin{matrix}0& &0\\ &\ddots& \\0& &0\end{matrix}$}
\sbox4{$\begin{matrix}\frac{-(a^2+b^2)}{a}& &0\\ &\ddots& \\0& &\frac{-(a^2+b^2)}{a}\end{matrix}$}
Det\left[
\begin{array}{c|c}
\usebox{0}&\usebox{1}\\
\hline
  \vphantom{\usebox{1}}\usebox{1}&\usebox{2}
\end{array}
\right]&= (x^2-(a^2+b^2))^n
\end{split}
\end{align}
Therefore the eigenvalues of $H^{(g\circ P)\mid_D(0)}$ are $\sqrt{a^2+b^2}$ and $-\sqrt{a^2+b^2}$, with algebraic multiplicity $n$ each. That is, all the critical points of $g\circ P$ arising from the critical points of $P$ have index $n$.

However, the preimages of the index 0  and index $2$ critical points $c^0$ and $c^2$ of $g$, are hyperplane sections each. Hence the function $g\circ P$ has non-isolated critical points on $Y$.

\subsection{Lefschetz Hyperplane Section Theorem}
\label{sec:Lefschetz}
Let $X$ be an n-dimensional smooth variety and $Y$ it's blowup along $X'$. Let $P:Y\to \mathbb{P}^1$ be the Lefschetz Pencil. Let $D_{c^0}$ be a closed real-dimension 2 disc with boundary around $c^0$ that does not contain any critical values of the Pencil, explicitly, if $y_1,y_2\cdots y_\mu$ are all the critical points of $P$, then choose an $\epsilon< (g\circ P)(y_i)$, for all $i\in\{1,2,\cdots,\mu\}$ and $D_{c^0}$ is defined as $g^{-1}(-\infty,\epsilon]$. For a point $p\in \mathbb{P}^1$, the fiber of $P$ at $p$ is defined as $Y_p=P^{-1}(p)$. Let $B_0$ be the preimage of $D_{c^0}$. Since $D_{c^0}$ is closed and does not contain any critical values of the Pencil, we have that preimage $P^{-1}(D_{c^0})$ is closed and does not contain any critical points of the Pencil. 

Similarly there exists a closed real-dimension 2 disc, $D_{c^2}=g^{-1}[1-\epsilon,\infty)$ containing $c^2$ which does not contain any critical points of the pencil $P$ and let $B_2=P^{-1}({D}_{c^2})$.

FIGURE \ref{fig}. is an illustration of the map
$Y\stackrel{\mathrm{{P}}}{\longrightarrow} {\mathbb{P}_1}\stackrel{\mathrm{\emph{g}}}{\longrightarrow}{\mathbb{R}}$
\begin{figure}[h]
\centerline{\includegraphics{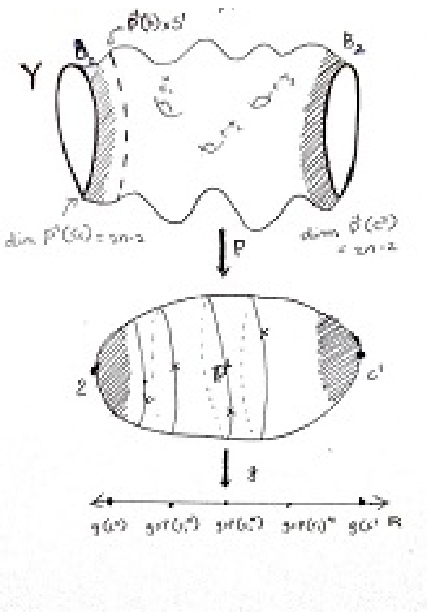} }
\caption{}
\label{fig}
\end{figure}
\begin{lemma}Let $B_0$, $B_2$ be the preimages of $D_{c^0}$, $D_{c^2}$ respectively. Let $\partial B_0$, $\partial B_2$ denote the boundary of $B_0$ and $B_2$, then
\[B_0\cong P^{-1}(c^0)\times D_{c^0}\cong P^{-1}(c^2)\times D_{c^2}\cong B_2\]and
\[\partial B_0 \cong P^{-1}(c^0)\times S^1\cong P^{-1}(c^2)\times S^1\cong \partial B_2,\]
\end{lemma}
\begin{proof}
Since $D_{c^0}$ does not contain any critical points of the Pencil, by Ehresmann's fibration Theorem, $P$ fibers $B_0$ locally trivially. Therefore, \[B_0\cong P^{-1}(D_{c^0})\cong P^{-1}({c^0})\times D_{c^0}\]
Let $b_1\in \partial D_{c^0}$. By Ehresmann's fibration Theorem,
 \[P^{-1}(\partial D_{c^0})\cong P^{-1}(b_1)\times S^1.\] Therefore,
\[
\partial(B_0)\cong \partial (P^{-1}(D_{c^0}))\cong  (P^{-1}(\partial D_{c^0}))\cong P^{-1}(b_1)\times S^1.
\]
Since there does not exist any critical points of the Pencil in $D_{c^0}$, $P^{-1}(b_1)\cong P^{-1}({c^0})$ and
\begin{equation*}\partial(B_0)\cong P^{-1}({c^0})\times S^1.\end{equation*}
Similarly for $B_2$ and $\partial B_2$.
\end{proof}

Since $B_0$ does not contain any critical points of $P$, $B_0$ is homotopic to $Y_{c^0}$, similarly $Y_{c^2}$ is homotopic to $B_2$.
\begin{theorem}
\label{RH}
Let $H_\lambda(Y\symbol{92}\mathring{B_2},B_0)$ denote the homology of $Y\symbol{92}\mathring{B_2}$ relative to $B_0$ and $H_\lambda(Y\symbol{92}\mathring{B_0},B_2)$ denote the homology of $Y\symbol{92}\mathring{B_0}$ relative to $B_2$, then \[H_\lambda(Y\symbol{92}\mathring{B_2},B_0)\cong H_\lambda(Y\symbol{92}\mathring{B_0},B_2)=\begin{cases}
\mathbb{Z}^\mu  &\quad \lambda=n\\
0  &\quad \text{otherwise} \\
\end{cases}\]where $\mu$ is the number of critical points of the Pencil.
\end{theorem}
\begin{proof}
Since $Y\symbol{92}\mathring{B_2}$ is the manifold $Y_{g\circ P}^{1-\epsilon}=(g\circ P)^{-1}(-\infty,1-\epsilon]$ and $B_0$ is the manifold $Y_{g\circ P}^\epsilon=(g\circ P)^{-1}(-\infty,\epsilon]$, from \ref{morsehomology}, $Y\symbol{92}\mathring{B_2}$ is homotopic to $B_0$ with $\mu$ number of n-cells attached.
 Therefore,
 \[H_\lambda(Y\symbol{92}\mathring{B_2},B_0)= \begin{cases}
      \mathbb{Z}^\mu  &\quad \lambda=n\\
     
      0  &\quad \text{otherwise} \\
     \end{cases}\]
Similarly, $Y\symbol{92}\mathring{B_0}$ is the manifold $Y_{-g\circ P}^{1-\epsilon}=(-g\circ P)^{-1}(-\infty,1-\epsilon]$ and $B_2$ is the manifold $Y_{-g\circ P}^\epsilon=(-g\circ P)^{-1}(-\infty,\epsilon]$ and from Morse Theory~\cite{milnor2016morse}, $Y\symbol{92}\mathring{B_0}$ is homotopic to $B_2$ with $\mu$ number of n-cells attached. Therefore,
\[H_\lambda(Y\symbol{92}\mathring{B_0},B_2)= \begin{cases}
      \mathbb{Z}^\mu  &\quad \lambda=n\\
     
      0  &\quad \text{otherwise} \\
     \end{cases}\]
\end{proof}
\begin{theorem}Recall that $Y$ is the blow-up variety of $X$ of dimension $n$ and $B_0=P^{-1}(D_{c^0})$. Then,
$H_{\lambda+1}(Y,B_0)\cong H_{\lambda-1}(B_0)$, for $\lambda \leq n$.
\end{theorem}
\begin{proof} 
Consider the homology long exact sequence of the triple given by $B_0\subset Y\symbol{92}\mathring{B} \subset Y$,
\begin{equation}\label{AYminuB}
\begin{tikzcd}
\cdots \arrow[r] & {H_{\lambda}(Y\symbol{92}\mathring{B_2},B_0)}\arrow[r] & {H_{\lambda}(Y,B_0)} \arrow[r, "" ] & {H_{\lambda}(Y,Y\symbol{92}\mathring{B_2})}\arrow[r, "" ]  &  {H_{\lambda-1}(Y\symbol{92}\mathring{B_2},B_0)}\arrow[r]& \cdots\\
\end{tikzcd}
\end{equation}
We have that \begin{equation}\label {eq:7}
\begin{split}
H_\lambda(Y, Y\symbol{92}\mathring{B_2})&\cong H_\lambda(B_0, \partial B_0)\cong H_\lambda(X_{c^0}\times D_1, X_{c^0}\times S^1)\\
&\cong H_{\lambda}(X_{c^0}\times (D_1, S^1))\\
&\cong \bigoplus\limits_{i}H_i(X_{c^0})\otimes H_{\lambda-i}(D_1, S^1)\\
&\cong H_{\lambda-2}(X_{c^0})\\
&\cong H_{\lambda-2}(B_0).
\end{split}
\end{equation}
The first isomorphism is through excision property. Since $B_0$ do not have any critical points of the pencil, Ehresmann's fibration theorem gives us the second isomorphism 
An application of the Künneth formula gives the fourth isomorphism and the final isomorphism follows from the following fact that
\begin{equation}H_q(D_1,S^1)=
\begin{cases}R  &\quad q=2\\     
     0  &\quad \text{otherwise} \\
    \end{cases}
\end{equation}
Substituting $(\ref{eq:7})$ and Theorem \ref{RH} in (\ref{AYminuB}), we have for $\lambda< n$
\begin{equation*}\label {Eq: 82}
\begin{tikzcd}
{0}\arrow[r] & {H_{\lambda}(Y,B_0)} \arrow[r, "" ] & {H_{\lambda-2}(B_0)}\arrow[r, "" ]  &  {0}\\
\end{tikzcd}\end{equation*}
Therefore,
\begin{equation}
\label{YrelA}
H_\lambda(Y,B_0)\cong H_{\lambda-2}(B_0)
\end{equation}
\end{proof}
\begin{theorem}
\label{E}
Let $\nu$ be a continuous mapping between pairs of compact Euclidean
neighborhood retracts, such that $\nu:X\symbol{92}X_1 \to Y\symbol{92}Y_1$ is a homeomorphism, then $\nu$ induces an isomorphism $\nu_{\lambda}:H_\lambda(X,X_1) \to H_\lambda(Y,Y_1)$.
\end{theorem}
\begin{proof}
This can be proved using the Neighborhood theorem and excision property. Homology commutes with direct limit.
\end{proof}
Since we deal with manifolds, ENR property follows.   
From what we have about the homology of the modification, we are good to start finding the homology of the variety $X$ itself.

\begin{proof}[\textbf{Proof of \ref{LHST}}]
From the long exact sequence of the pair $X_{c^0}, X$ given by
\begin{equation*}
\begin{tikzcd}
\cdots \arrow[r] & {H_{\lambda +1}(X,X_{c^0})}\arrow[r] & {H_{\lambda}(X_{c^0})}\arrow[r, ""] &
{H_{\lambda}(X)}\arrow[r] & {H_{\lambda}(X,X_{c^0})} \arrow[r]& 
\cdots
\end{tikzcd}
\end{equation*} we get that the statement of the theorem is equivalent to saying that $H_\lambda(X, X_{c^0})=0$ for $\lambda \leq n-1.$
\\

Consider the homology long exact sequence of the triple $B_0\subset Y'\cup B_0\subset Y$, 
\begin{equation}\label{impl}
\begin{tikzcd}
\cdots \arrow[r] & {H_{\lambda}(Y'\cup B_0,B_0)}\arrow[r] & {H_{\lambda}(Y,B_0)} \arrow[r, "" ] & {H_{\lambda}(Y,Y'\cup B_0)}\arrow[r, "" ]  &  {H_{\lambda-1}(Y'\cup B_0,B_0)}\arrow[r]& \cdots\\
\end{tikzcd}
\end{equation}
Consider our third term of the sequence, 
\begin{equation}
\label{H}
\begin{split}
 H_\lambda(Y, Y'\cup B_0)&\cong H_\lambda(Y, Y'\cup Y_{c^0})
 \end{split}
\end{equation}
\\

Since $Y\symbol{92}(Y'\cup Y_{c^0})\cong X\symbol{92}X_{c^0}$, by applying Theorem \ref{H} in (\ref{H}),
\begin{equation}
\label{YYstickA}
\begin{split}
 H_\lambda(Y, Y'\cup B_0)&\cong H_\lambda(X, X_{c^0})
 \end{split}
\end{equation}
We aim to prove this term is zero. We already know from (\ref{YrelA}) that the second term $H_{\lambda}(Y,B_0)\cong H_{\lambda-2}(B_0)$. And for the first term
\begin{equation*}
\begin{split}
  H_\lambda(Y'\cup B_0,B_0)&\cong H_\lambda(Y',X'),\quad \text{by theorem \ref{E}}\\
  &\cong H_\lambda(X'\times \mathbb{P}^1,X')\\
\end{split}
\end{equation*}
By Künneth Formula, 
\begin{equation}
\label{blowax}
\begin{split}
  H_\lambda(X'\times \mathbb{P}^1,X')&\cong H_\lambda(X'\times (\mathbb{P}^1,1))\\
  &\cong \bigoplus_iH_i(X')\otimes H_{\lambda-i}(\mathbb{P}^1)\\
  &\cong H_{\lambda-2}(X')
\end{split}
\end{equation}
Substituting (\ref{YrelA}), (\ref{YYstickA}), (\ref{blowax}) in (\ref{impl}), we get 
\begin{equation}
\label{K}
\begin{tikzcd}
0 \arrow[r] & {H_{\lambda-2}(X')}\arrow[r, "K"] &
{H_{\lambda-2}(X_{c^0})}\arrow[r] & {H_{\lambda}(X,X_{c^0})} \arrow[r]& 
0
\end{tikzcd}
\end{equation}
Since $X'=X_{c^0}\cap B_0$ is a smooth hyperplane section of $X_{c^0}$, $K$ is an isomorphism  by induction and therefore, by the exactness of the sequence (\ref{K})
\begin{equation}H_{\lambda}(X,X_{c^0})=0,\quad \text{for $\lambda \leq n-1$}\end{equation}
\end{proof}

\bibliography{abc}{}
\bibliographystyle{siam}

\end{document}